\documentclass[12pt]{amsart}
\usepackage{amsfonts}
\usepackage{amsmath}
\usepackage{amssymb}
\usepackage{mathtools}
\usepackage{mathrsfs}

\newtheorem{theorem}{Theorem}[section]
\newtheorem{lemma}[theorem]{Lemma}

\newtheorem{corollary}[theorem]{Corollary}

\theoremstyle{definition}
\newtheorem{definition}[theorem]{Definition}

\theoremstyle{remark}
\newtheorem{remark}[theorem]{Remark}

\makeatletter
\newcommand{\range}{\mathop{\operator@font range}}
\newcommand{\rank}{\mathop{\operator@font rank}}
\newcommand{\dom}{\mathop{\operator@font dom}}
\newcommand{\Real}{\mathop{\operator@font Re}}
\newcommand{\cont}{\mathop{\operator@font Cont_w}}
\newcommand{\alg}{\mathop{\operator@font Alg\mathcal{N}}}
\newcommand{\tr}{\mathop{\operator@font tr}}
\newcommand{\sgn}{\mathop{\operator@font sgn}}
\newcommand{\kn}{\mathop{\operator@font \mathcal{K}(\mathcal{N})}}
\newcommand{\rad}{\mathop{\operator@font Rad(\mathcal{N})}}
\makeatother

\begin{document}
\title{Topological Radicals of Nest Algebras}
\author{G. Andreolas}
\author{M. Anoussis}
\address{Department of Mathematics, University of the Aegean, 832\,00
Karlovassi, Samos, Greece}
\email{gandreolas@aegean.gr}
\address{Department of Mathematics, University of the Aegean, 832\,00
Karlovassi, Samos, Greece}
\email{mano@aegean.gr}

\begin{abstract}
 Let $\mathcal{N}$ be a nest on a Hilbert space $H$ and $\alg$ the corresponding nest algebra. We determine the
hypocompact  radical of  $\alg$. Other topological radicals are also characterized.
\end{abstract}

\maketitle

\section{Introduction}

 Let $\mathcal{A}$ be a Banach algebra. The  
 \textit{Jacobson 
radical}  of $\mathcal{A}$ 
 is  defined as the intersection of the kernels of the algebraically irreducible representations of  $\mathcal{A}$.
 A topologically irreducible representation of
$\mathcal{A}$ is a continuous homomorphism of $\mathcal{A}$
into the Banach algebra of bounded linear operators
on  a  Banach  space $X$ for which   no  nontrivial,  closed  subspace of $X$ is invariant. 
It has been shown in \cite{dix} that  the intersection of the kernels of these 
representations is in a reasonable sense a new radical that can be strictly smaller than the Jacobson radical.

The theory of topological radicals of Banach algebras originated with Dixon \cite{dix} in order to study this new radical as well as other
radicals associated with various types of representations.

Shulman and Turovskii 
have further developed the theory of topological radicals in a series of papers \cite{shu0,  shu1, shu2, shu3, shu4, shu5} and applied it to the study of various problems of Operator Theory and 
Banach algebras. They introduced many new topological radicals. Among them there are the hypocompact radical, the hypofinte 
radical and the scattered radical. 
 These radicals are closely related to the theory of elementary operators on Banach algebras \cite{bre, shu2}.
 
 Let us  recall Dixon's definition of topological radicals.

\begin{definition}
 A \textit{topological radical} is a map $\mathcal{R}$ associating with each Banach algebra $\mathcal{A}$ a closed ideal 
$\mathcal{R}(\mathcal{A}) \subseteq\mathcal{A}$ such that the following hold.
\begin{enumerate}
 \item $\mathcal{R}(\mathcal{R}(\mathcal{A}))=\mathcal{R}(\mathcal{A})$.
 \item $\mathcal{R}(\mathcal{A}/\mathcal{R}(\mathcal{A}))\!=\!\{0\}$, where $\{0\}$ denotes the zero coset in 
$\mathcal{A}/\mathcal{R}(\mathcal{A})$.
 \item If $\mathcal{A}$, $\mathcal{B}$ are Banach algebras and $\phi:\mathcal{A}\to\mathcal{B}$ is a continuous 
epimorphism, then $\phi(\mathcal{R}(\mathcal{A}))\subseteq \mathcal{R}(\mathcal{B})$.
 \item If $\mathcal{I}$ is a closed ideal of $\mathcal{A}$, then $\mathcal{R}(\mathcal{I})$ is a closed ideal of $\mathcal{A}$ 
and $\mathcal{R}(\mathcal{I})\subseteq \mathcal{R}(\mathcal{A})\cap\mathcal{I}$.
\end{enumerate}
\end{definition}

 An element $a$ of a Banach  
algebra $\mathcal{A}$ is said to be \textit{compact}  if the map $M_{a,a}:\mathcal{A}\rightarrow\mathcal{A}$, $x\mapsto 
axa$ is 
compact. Following Shulman and Turovskii \cite{shu4} we will call a  Banach algebra 
$\mathcal{A}$  \textit{hypocompact}  if any nonzero quotient $\mathcal{A}/\mathcal{J}$ by a closed ideal 
$\mathcal{J}$ contains a nonzero
compact element. Shulman and Turovskii have proved that any Banach  algebra $\mathcal{A}$ has a largest hypocompact 
ideal \cite[Corollary 3.10]{shu4} which is denoted by $\mathcal{R}_{hc}(\mathcal{A})$ and that   the map $\mathcal{A}\to 
\mathcal{R}_{hc}(\mathcal{A})$ is a topological radical \cite[Theorem 3.13]{shu4}. 
 The ideal $\mathcal{R}_{hc}(\mathcal{A})$
 is called the hypocompact  radical of $\mathcal{A}$.

If  $X$ is a Banach space, we shall denote by $\mathcal{B}(X)$ the Banach algebra of all bounded operators on $X$ and by
$\mathcal{K}(X)$ the Banach algebra of all compact operators on $X$. 
Vala has shown in \cite{1964} that if $X$ is a Banach space, an element $a \in \mathcal{B}(X)$ is a compact element if and only if $a \in \mathcal{K}(X)$.  
Since by \cite[Lemma 8.2]{bre} the compact elements are always contained in the hypocompact radical, we obtain  $\mathcal{K}(X)\subseteq 
\mathcal{R}_{hc}(\mathcal{B}(X))$. It follows that if 
   $H$ is a separable Hilbert space,  the hypocompact radical of $\mathcal{B}(H)$ is $\mathcal{K}(H)$. 
 Indeed, the ideal  $\mathcal{K}(H)$ is the only proper ideal of $\mathcal{B}(H)$ while the Calkin algebra 
$\mathcal{B}(H)/\mathcal{K}(H)$ does not have any non-zero compact element \cite[section 5]{fs}.
 
  Shulman and Turovskii observe   in \cite[p. 298]{shu4} that  there exist Banach spaces $X$, such that the hypocompact radical
 $\mathcal{R}_{hc}(\mathcal{B}(X))$ of $\mathcal{B}(X))$ contains all the weakly compact
 operators and  contains strictly 
 the ideal of compact operators $\mathcal{K}(X)$.  

  Argyros and Haydon construct in \cite{ah}   a Banach space $X$ 
 such that every operator in $\mathcal{B}(X)$ is  a scalar 
multiple of the identity plus a compact operator. In that case,  it follows that $\mathcal{B}(X)/\mathcal{
K}(X)$ is finite-dimensional and hence 
$\mathcal{R}_{hc}((\mathcal{B}(X))=\mathcal{B}(X)$. 

In this paper we characterize the hypocompact radical of a nest algebra.
Nest algebras form a class
of non-selfadjoint operator algebras that generalize the block upper triangular matrices to an infinite dimensional 
Hilbert
space context. They were introduced by Ringrose in \cite{ring} and since then, they have been studied by many authors. 
The monograph of Davidson \cite{dav}
is recommended as a reference.

The  ideal structure  of nest algebras has an important part in the development of the  theory of nest algebras.  
Ringrose characterized the Jacobson radical of a nest algebra in \cite[Theorem 5.3]{ring}. 
Moreover, it  follows from \cite[Theorems 4.9 and 5.3]{ring} that the intersection of the kernels of  the  
topologically 
irreducible representations of a nest algebra coincides with the Jacobson radical. 

We introduce now some definitions and notations that we will use in the sequel. 
A nest $\mathcal{N}$ is a totally
ordered family of closed subspaces of a Hilbert space $H$ containing $\{0\}$ and $H$, which is closed under intersection 
and
closed span. If $H$ is a Hilbert space and $\mathcal{N}$ a nest on $H$, then the nest algebra $\alg$ is the algebra of 
all
operators $T\in\mathcal{B}(H)$ such that $T(N)\subseteq N$ for all $N\in\mathcal{N}$. We shall usually denote both the 
subspaces belonging to a nest and their corresponding orthogonal projections by the same symbol. If 
$(N_{\lambda})_{\lambda\in\Lambda}$ is a family 
of
subspaces of a Hilbert space, we denote by $\vee \{N_{\lambda}:\lambda\in\Lambda\}$ their closed linear span and by
$\wedge\{N_{\lambda}:\lambda\in\Lambda\}$ their intersection. If $\mathcal{N}$ is a nest and $N\in\mathcal{N}$, then
$N_-=\vee\{N^{\prime}\in\mathcal{N}:N^{\prime}<N\}$.
Similarly we define $N_+=\wedge\{N^{\prime}\in\mathcal{N}:N^{\prime}>N\}$. The subspaces $N\cap N_{-}^{\perp}$ are 
called
the \textit{atoms} of $\mathcal{N}$. 
If $e,f$ are elements of a
Hilbert space $H$, we denote by $e\otimes f$ the rank one operator on $H$ defined by
$(e\otimes f)(h)=\langle h,e\rangle f.$
We shall
frequently use the fact that a rank one operator $e\otimes f$ belongs to a nest algebra, $\alg$, if and only if there 
exists an
element $N$ of $\mathcal{N}$ such that $e\in N_-^{\perp}$ and $f\in N$ \cite[Lemmas 2.8 and 3.7]{dav}.
Throughout the paper we denote by $\mathcal{N}$ a nest 
acting
on a Hilbert space $H$ and by $\mathcal{K}(\mathcal{N})$ the ideal of compact operators of $\alg$.
In addition, all ideals are considered to be closed. 
The radical of a nest algebra $\alg$ will be denoted by $\rad$. 
The following is 
\cite[Theorem 5.3]{ring}.

\begin{theorem}[Ringrose's Theorem]\label{ringrose} Let $\mathcal{N}$ be a nest on a Hilbert space $H$. The Jacobson 
radical of $\alg$ coincides with the set of operators 
$a\in\alg$ for which the quantities $\inf\{\|PQ^{\perp}aPQ^{\perp}\|:P\in\mathcal{N}; P>Q\}$ and 
$\inf\{\|QP^{\perp}aQP^{\perp}\|:P\in\mathcal{N}; P<Q\}$ are zero for all $Q\in\mathcal{N}$.
\end{theorem}

\section{Main Result}

\begin{lemma} \label{cont}
 Let $\mathcal{N}$ be a nest on a Hilbert space $H$ and $Q\in\mathcal{N}$ such that $Q_-=Q$. Suppose that $a,b\in\alg$ 
such 
that 
$\|QP^{\perp}aQP^{\perp}\|\geq2\varepsilon$ and $\|QP^{\perp}bQP^{\perp}\|\geq2\varepsilon$ for some 
$\varepsilon>0$ and for all $P<Q$, $P\in\mathcal{N}$. Then, there exist orthonormal sequences $(e_n)$, $(f_n)$ such 
that $e_n\otimes 
f_n\in\alg$ and $\|QP^{\perp}a(\sum_{n=1}^{\infty}e_{k_n}\otimes f_{k_n}) bQP^{\perp}\|\geq\varepsilon^2$ for all 
$P<Q$, $P\in\mathcal{N}$ 
and for any strictly increasing sequence $(k_n)\subseteq \mathbb{N}$.
\end{lemma}
\begin{proof}
Let $(R_n)_{n\in\mathbb{N}}$ be a strictly increasing sequence that is sot convergent to $Q$.
 We set $P_1=R_1$. Then, $\|QP_1^{\perp}aQP_1^{\perp}\|\geq 2\varepsilon$. We choose a norm one vector $f'_1\in Q$ such 
that $\|QP_1^{\perp}aQP_1^{\perp}(f'_1)\|\geq\frac{3}{2}\varepsilon$. Then, we choose a projection $P_2=R_{k_2}$ 
with $k_2>1$, such that 
$$\|QP_1^{\perp}aQP_1^{\perp}P_2(f'_1)\|\geq\varepsilon.$$
We set $f_1=\frac{1}{\|P_1^{\perp}P_2f'_1\|}P_1^{\perp}P_2f'_1$. Then  $f_1\in 
P_1^{\perp}P_2$, $\|f_1\|=1$ and  
$$\|P_2P_1^{\perp}aP_2P_1^{\perp}(f_1)\|\geq\varepsilon.$$
Suppose that there exist $P_3=R_{k_3},\ldots,P_n=R_{k_n}$, where $k_2<\ldots<k_n$ such that 
$\|P_iP_{i-1}^{\perp}aP_iP_{i-1}^{\perp}(f_{i-1})\|\geq\varepsilon$  for some orthonormal vectors $(f_i)_{i=1}^n$, 
where 
$f_{i-1}\in P_iP_{i-1}^{\perp}$, $i\in\{3,\ldots,n\}$. Given that $\|QP_n^{\perp}aQP_n^{\perp}\|\geq2\varepsilon$, we 
consider the arguments of the first step of the proof to obtain  a 
projection $P_{n+1}=R_{k_{n+1}}$ for some $k_{n+1}>k_n$ and a norm one vector $f_n\in P_{n+1}P_n^{\perp}$ such that
$$\|P_{n+1}P_n^{\perp}aP_{n+1}P_n^{\perp}(f_n)\|\geq\varepsilon.$$
Note that $(P_n)_{n\in\mathbb{N}}$ is sot convergent to $Q$ as a subsequence of $(R_n)_{n\in\mathbb{N}}$.

In the same way, we can find a subsequence $(S_n)_{n\in\mathbb{N}}$ of $(R_n)_{n\in\mathbb{N}}$ such that 
$S_n>P_n$ for all $n\in\mathbb{N}$ and   an orthonormal sequence 
$(e_n)_{n\in\mathbb{N}}\subseteq H$ such that  $e_n\in S_{n+1}S_n^{\perp}$  and
$$\|S_{n+1}S_n^{\perp}b^*S_{n+1}S_n^{\perp}(e_n)\|\geq\varepsilon,$$
 for all 
$n\in\mathbb{N}$. It follows that $e_n\otimes f_n\in\alg$. Let $P\in\mathcal{N}$ and $i\in\mathbb{N}$ such that 
$P_{i+1},S_{i+1}>P$. Then,
\begin{eqnarray*}
 \left\|QP^{\perp}a\left(\sum_{n\in\mathbb{N}}e_n\otimes f_n\right) b QP^{\perp}\right\| &\geq& 
\left\|P_{i+1}P_{i}^{\perp}a\left(\sum_{n\in\mathbb{N}}e_n\otimes f_n\right) bS_{i+1}S_{i}^{\perp}\right\|\\
&=& \left\|\sum_{n\in\mathbb{N}}S_{i+1}S_{i}^{\perp}b^*(e_n)\otimes P_{i+1}P_{i}^{\perp}a(f_n)\right\|\\
&=& \left\|S_{i+1}S_{i}^{\perp}b^*(e_i)\otimes P_{i+1}P_{i}^{\perp}a(f_i)\right\|\\
&=& \left\|S_{i+1}S_{i}^{\perp}b^*(e_i)\right\|\left\|P_{i+1}P_{i}^{\perp}a(f_i)\right\|\\
&\geq& \varepsilon^2.
\end{eqnarray*}
The proof is identical  for any strictly increasing sequence $(k_n)\subseteq \mathbb{N}$.
\end{proof}

In the proof of Theorem \ref{maingeneral} we shall use the following fact which is implied by \cite[Proposition 
1.18]{dav} and Ringrose's Theorem.

\begin{lemma}\label{utility}
 Let $Q\in\mathcal{N}$, $Q=Q_-$ and $a\in\kn+\rad$. Then, $\inf\{\|QP^{\perp}aQP^{\perp}\|: P\in\mathcal{N}; P<Q\}=0$.
\end{lemma}

\begin{remark} \label{Q+}
 Similar statements as those of Lemmas (\ref{cont}) and (\ref{utility}) hold in the case that $Q_+=Q$.
\end{remark}

\begin{theorem} \label{maingeneral}
The hypocompact radical of $\alg$ is the ideal $\kn+$ $\rad$.
\end{theorem}
\begin{proof}
The ideal generated by the compact elements of $\alg$ is the ideal $\kn+\rad$ \cite[Theorem 3.2]{anan} and therefore
$\kn+\rad$ $\subseteq \mathcal{R}_{hc}(\mathcal{N})$ \cite[Lemma 8.2]{bre}. Let $\mathcal{J}$ 
be an ideal of $\alg$ strictly larger that $\kn+\rad$ and $a\in\mathcal{J}\setminus (\kn+\rad)$. It suffices 
to show 
that the element $\varphi(a)\in \mathcal{J}/(\kn+\rad)$ is not compact, where $\varphi:\mathcal{J}\rightarrow 
\mathcal{J}/(\kn+\rad)$ is the quotient map. We consider the sets 
$$\mathfrak{Q}_-\!=\!\{Q\!\in\!\mathcal{N}\!:\! Q \neq \{0\}, \exists \epsilon_Q>0 : \|QP^{\perp}aQP^{\perp}\|\geq 
2\epsilon_Q \; \forall  P\!\in\!\mathcal{N}; P<Q\}$$
and
$$\mathfrak{Q}_+\!=\!\{Q\!\in\!\mathcal{N}: Q \neq H, \exists \epsilon_Q>0 : \|Q^{\perp}PaQ^{\perp}P\|\geq 
2\epsilon_Q\; \forall  P\!\in\!\mathcal{N}; P>Q\}.$$
From Ringrose's Theorem 
it 
follows that the set $\mathfrak{Q}=\mathfrak{Q}_-\cup\mathfrak{Q}_+$ is not empty. We distinguish two cases:
\begin{enumerate}
 \item Firstly, we suppose that there exists a projection $Q\in\mathfrak{Q}_-$ such that 
$Q_-=Q$. If we assume that 
there exists a $Q\in\mathfrak{Q}_+$ such that $Q=Q_+$ the proof is similar (see Remark 
(\ref{Q+})).
Applying Lemma \ref{cont} 
we obtain two orthonormal sequences $(h_n)$, $(g_n)$ such that $h_n\otimes g_n\in\alg$  
for all $n\in\mathbb{N}$ and 
$\left\|QP^{\perp}a\left(\sum_{n\in\mathbb{N}} h_n\otimes g_n\right)\right.$ 
$\left.aQP^{\perp}\right\|\geq \epsilon_Q^2$ for all 
$P<Q$. 
We set $x=\left(\sum_{n\in\mathbb{N}} h_n\otimes g_n\right) 
a\in\mathcal{J}$. Let $\epsilon>0$ be such that $2\epsilon=\min\{2\epsilon_Q, 
\epsilon_Q^2\}$. Applying again Lemma \ref{cont} to the 
operators $ax=a\left(\sum_{n\in\mathbb{N}} h_n\otimes 
g_n\right) a$ and $a$ we obtain orthonormal sequences $(e_n)$ and $(f_n)$ such that $e_n\otimes f_n\in\alg$ and 
$\|QP^{\perp}ax\left(\sum_{n\in\mathbb{N}}e_{k_n}\otimes f_{k_n}\right)aQP^{\perp}\|\geq \epsilon^2$ for 
all 
$P<Q$ and for any strictly increasing sequence $(k_n)\subseteq\mathbb{N}$. Let $(A_n)_{n\in\mathbb{N}}$ 
be a partition of $\mathbb{N}$ such that $A_n$ is an infinite set for all
$n\in\mathbb{N}$. We set $B_n=\cup_{i=1}^{n} A_i$. Note that $\|\sum_{i\in C}e_i\otimes x(f_i)\|\leq \|a\|$ for any 
subset $C$ of $\mathbb{N}$.

Now, we shall prove that the
sequence 
$$\left(\varphi\left( a\left(\sum_{i\in B_n}e_i\otimes x(f_i)\right)a\right)\right)_{n\in\mathbb{N}}\subseteq 
\mathcal{J}/(\kn+\rad)$$
does not have any Cauchy subsequence. Indeed, for any $l,m\in\mathbb{N}$ with $l> m$:
\begin{multline*}
\left\|\varphi(a)\varphi\left(\sum_{i\in B_l} e_i\otimes
x(f_i)\right)\varphi(a)\right. \\ \left.-\varphi(a)\varphi\left(\sum_{j\in B_m} e_j\otimes
x(f_j)\right)\varphi(a)\right\|_{\scriptscriptstyle\mathcal{J}/(\kn+\rad)}
\end{multline*}
\begin{eqnarray*}
\hspace{5ex}&=& \inf_{\scriptscriptstyle r\in(\kn+\rad)}\left\|a\left(\sum_{i\in B_l-B_m}e_i\otimes 
x(f_i)\right)a+r\right\| \\
&\geq& \left\|a\left(\sum_{i\in B_l-B_m}e_i\otimes x(f_i)\right)a+r_{\epsilon}\right\|-\frac{\epsilon^2}{4},
\end{eqnarray*}
for some $r_{\epsilon}\in (\kn+\rad)$.

There exists a projection $P<Q$ such that $\|QP^{\perp}r_{\epsilon}QP^{\perp}\|<\frac{\epsilon^2}{4}$ (Lemma 
\ref{utility}). Therefore,
\begin{eqnarray*}
 & &\left\|a\left(\sum_{i\in B_l-B_m}e_i\otimes x(f_i)\right)a+r_{\epsilon}\right\|-\frac{\epsilon^2}{4}\\
 &\geq& \left\|QP^{\perp}\left(a\left(\sum_{i\in B_l-B_m}e_i\otimes 
x(f_i)\right)a+r_{\epsilon}\right)QP^{\perp}\right\|-\frac{\epsilon^2}{4}\\
&\geq& \left\|QP^{\perp}a\left(\sum_{i\in B_l-B_m}e_i\otimes 
x(f_i)\right)aQP^{\perp}\right\|-\left\|QP^{\perp}r_{\epsilon}QP^{\perp}\right\|-\frac{\epsilon^2}{4}\\
&\geq& \epsilon^2-\frac{\epsilon^2}{4}-\frac{\epsilon^2}{4}=\frac{\epsilon^2}{2}.
\end{eqnarray*}
Thus, $\varphi(a)$ is a non-compact element of $\mathcal{J}/(\kn+\rad)$.

 \item Secondly, we suppose that for all $Q\in\mathfrak{Q}_-$ we have that $Q_-<Q$ and for all $Q\in\mathfrak{Q}_+$ we 
have that $Q_+>Q$. In that case, we shall consider the set $\mathfrak{Q}_-$ instead of $\mathfrak{Q}$ since 
$Q_+\in\mathfrak{Q}_-$ for all $Q\in\mathfrak{Q}_+$. Then the set 
$\mathfrak{E}_n=\{Q:\|QQ_-^{\perp}aQQ_-^{\perp}\|>\frac{1}{n}\}$ is finite for all $n\in\mathbb{N}$. Observe, that if 
the set $\mathfrak{E}_n$ was infinite for some $n\in\mathbb{N}$, then there would be a projection $Q\in\mathfrak{Q}$ 
which is an accumulation point, i.e. either $Q=Q_-$ or $Q=Q_+$ which is against our assumption. Now, we suppose that 
the operator $QQ_-^{\perp}aQQ_-^{\perp}$ is compact for all $Q\in\mathfrak{Q}_-$ and we shall arrive at a 
contradiction. Indeed, since  $\mathfrak{E}_n$ is finite for all $n\in\mathbb{N}$, the series 
$\sum_{Q\in\mathfrak{Q}_-}QQ_-^{\perp}aQQ_-^{\perp}$  is norm convergent and therefore its limit 
belongs to $\kn$. It follows that $a-\sum_{Q\in\mathfrak{Q}_-}QQ_-^{\perp}aQQ_-^{\perp}\in\rad$ and therefore, 
$a\in\kn+\rad$ which is a contradiction. We thus have that there exists a $Q\in\mathfrak{Q}_-$ such that the operator 
$a_Q=QQ_-^{\perp}aQQ_-^{\perp}$ is not compact. It follows that  $\mathcal{B}(QQ_-^{\perp})\subseteq \mathcal{J}$.  We 
define the map:
$$i:\mathcal{B}(QQ_-^{\perp})/\mathcal{K}(QQ_-^{\perp})\rightarrow \mathcal{B}(QQ_-^{\perp})/(\kn+\rad)$$
$$x+\mathcal{K}(QQ_-^{\perp})\mapsto x+\kn+\rad.$$
This map is obviously well defined. 
Now, we see that the map $i$ is an isometric isomorphism. Indeed,
\begin{eqnarray*}
\|x+\mathcal{K}(QQ_-^{\perp})\|_{{\scriptscriptstyle \mathcal{B}(QQ_-^{\perp})/\mathcal{K}(QQ_-^{\perp})}} 
\hspace{-0.3cm}&=& 
\inf\{\|x+K\|:K\in\mathcal{K}(QQ_-^{\perp})\}\\
&=& \inf_{\mathclap{\substack{\scriptscriptstyle K\in\kn\\\scriptscriptstyle 
R\in\rad}}}\|QQ_-^{\perp}(x+K+R)QQ_-^{\perp}\|\\
&\leq& \inf_{\mathclap{\substack{\scriptscriptstyle K\in\kn\\\scriptscriptstyle R\in\rad}}}\|x+K+R\|\\
&=&\|x\!+\!\kn\!+\!\rad\|_{\scriptscriptstyle\mathcal{B}(QQ_-^{\perp})/(\kn\!+\!\rad)}
\end{eqnarray*}
and the opposite inequality is immediate since $\mathcal{K}(QQ_-^{\perp})\subseteq \kn$ $+\rad$. 
If $\varphi(a)$ is a compact element of $\mathcal{J}/(\kn\!+\!\rad\!)$, then $\varphi(a_Q)$ is a compact element of 
$\mathcal{B}(QQ_-^{\perp})/(\kn+\rad)$. Since   $i(a_Q+\mathcal{K}(QQ_-^{\perp}))=\varphi(a_Q)$ it follows from above 
that 
 $a_Q+\mathcal{K}(QQ_-^{\perp})$ is a compact element of $\mathcal{B}(QQ_-^{\perp})/\mathcal{K}(QQ_-^{\perp})$. 
From \cite{fs} we know that
 $\mathcal{B}(QQ_-^{\perp})/\mathcal{K}(QQ_-^{\perp})$ has no compact elements.
Hence, $\phi(a)$ is not a compact element of $\mathcal{J}/(\kn+\rad)$.

\end{enumerate}
\end{proof}

\begin{remark}
The hypocompact radical of $\alg$ coincides with the ideal generated by the compact elements of $\alg$.
\end{remark}

The following definitions and results are taken from  \cite{shu5}. An element $a$ of a Banach  
algebra $\mathcal{A}$ is said to be  finite rank if the map $M_{a,a}:\mathcal{A}\rightarrow\mathcal{A}$, $x\mapsto axa$ is 
 finite rank. A  Banach algebra 
$\mathcal{A}$ is called  hypofinite if any nonzero quotient $\mathcal{A}/\mathcal{J}$ by a closed ideal 
$\mathcal{J}$ contains a nonzero
 finite rank element.   A Banach 
 algebra $\mathcal{A}$ has a largest hypofinite
ideal  which is denoted by $\mathcal{R}_{hf}(\mathcal{A})$ and the map $\mathcal{A}\to 
\mathcal{R}_{hf}(\mathcal{A})$ is a topological radical \cite[2.3.6]{shu5}. The ideal 
 $\mathcal{R}_{hf}(\mathcal{A})$ is called the 
 hypofinite radical of $\mathcal{A}$.
A Banach  
  algebra  
 is called \textit{scattered} if the
spectrum of every element $a\in \mathcal{A}$ is finite or countable.  A   Banach algebra 
$\mathcal{A}$ has a largest scattered ideal denoted by $\mathcal{R}_{sc}(\mathcal{A})$  and the map $\mathcal{A}\to 
\mathcal{R}_{sc}(\mathcal{A})$ is a topological radical as well \cite[Theorems 8.10, 8.11]{shu5}. The ideal 
 $\mathcal{R}_{sc}(\mathcal{A})$ is called the 
 scattered radical of $\mathcal{A}$.

\begin{corollary} \label{sc}
 $\mathcal{R}_{hf}(\alg)=\mathcal{R}_{hc}(\alg)=\mathcal{R}_{sc}(\alg)$.
\end{corollary}
\begin{proof}
 For $P \in \mathcal{N}$  denote by $\mathcal{J}_P$ the ideal  $$\mathcal{J}_P=\{a \in \alg: a=PaP^{\perp}\}$$ of $\alg$. 
 Since $\mathcal{J}_P$   has trivial multiplication, it is a hypofinite ideal and  is contained in the 
 hypofinite radical of $\alg$. It follows from  \cite[Theorem 5.4]{ring} that 
 the Jacobson radical of $\alg$ is the closure of the linear span of the set $\cup_{P \in \mathcal{N}} \mathcal{J}_P$, 
 and hence it is contained in
 the hypofinite radical of $\alg$. 
 The Corollary now follows from   \cite[Theorem 8.15]{shu5} and Theorem \ref{maingeneral}.
\end{proof}

\begin{corollary}
 The  hypocompact radical of $\alg/\kn$ coincides with its scattered radical which in turn is equal to the ideal 
$(\rad$ $+\kn)/\kn$.
\end{corollary}
\begin{proof}
 It follows from  \cite[Corollary 3.9]{shu4} and Theorem \ref{maingeneral} that  $(\rad$ $+\kn)/\kn$  
 is the hypocompact radical of $\alg/\kn$. 
 
 It follows from   
\cite[Corollary 8.13]{shu5}, Theorem \ref{maingeneral} and 
Corollary \ref{sc} that the scattered radical of $\alg/\kn$   is $(\rad+\kn)/\kn$.

\end{proof}

\bibliographystyle{amsplain}

\end{document}